\documentclass{amsart}

\usepackage[foot]{amsaddr}

\usepackage[T1]{fontenc}
\usepackage[utf8]{inputenc}
\usepackage[USenglish]{babel}
\usepackage{textcase}

\usepackage[
	bookmarks=true,
	plainpages=false,
	linktocpage,
	colorlinks=true,
	citecolor=green!80!black,
	linkcolor=red!70!black,
	filecolor=magenta,
	urlcolor=magenta,
	breaklinks,
	pdfauthor={Martin Winter},
]{hyperref}

\usepackage{amsmath,amsthm}
\usepackage{calc, mathtools} 

\iftrue
\usepackage{amssymb} 
\else
\usepackage[charter,cal=cmcal]{mathdesign}
\fi

\usepackage{colortbl,color} 
\usepackage[dvipsnames]{xcolor}
\usepackage{centernot} 
\usepackage{array} 
\usepackage{enumitem,moreenum} 
\usepackage{cite} 
\usepackage[nameinlink,capitalize,noabbrev]{cleveref}
\usepackage{nicefrac}

\usepackage{tikz-cd} 

\usepackage[font=small,labelfont=bf]{caption}

\usepackage{blkarray}

\usepackage{inconsolata}
\usepackage{dsfont}

\newcommand{\RR}{\mathbb{R}}    
                   
\def\^#1{^{(#1)}}
\def\s^#1{^{\smash{(#1)}}}

\def\:{\colon}

\definecolor{NiceBlue}{rgb}{0.15, 0.2, 0.75}
\newcommand{\Def}[1]{\textcolor{NiceBlue!90}{\textit{#1}}}

\newcommand{\labelstyle}[1]{\upshape(\textit{#1})}
\newcommand{\mylabel}{\labelstyle{\roman*}}
\newenvironment{myenumerate}{\begin{enumerate}[label=\mylabel]}{\end{enumerate}}

\def\itm#1{{\labelstyle{\romannumeral#1\relax}}}

\newcommand{\nls}{\nolinebreak\space}

\newcommand{\freespace}{\kern.07em} 
\newcommand{\free}{\freespace\cdot\freespace}

\newcommand{\bs}[1]{\boldsymbol{#1}}

\newcommand{\msays}[1]{{\footnotesize\textcolor{red}{#1}}}
\newcommand{\TODO}{\msays{TODO}}

\theoremstyle{plain}  
\newtheorem{theorem}{Theorem}
\newtheorem{corollary}[theorem]{Corollary}
\newtheorem{lemma}[theorem]{Lemma}

\newtheorem{conjecture}[theorem]{Conjecture}

\theoremstyle{definition} 

\crefname{theorem}{Theorem}{Theorems}
\crefname{proposition}{Proposition}{Propositions}
\crefname{lemma}{Lemma}{Lemmas}
\crefname{corollary}{Corollary}{Corollaries}
\crefname{conjecture}{Conjecture}{Conjecture}
\crefname{remark}{Remark}{Remarks}
\crefname{example}{Example}{Examples}
\crefname{definition}{Definition}{Definitions}
\crefname{problem}{Problem}{Problems}
\crefname{observation}{Observation}{Observation}
\crefname{construction}{Construction}{Construction}

\DeclareMathOperator{\conv}{conv}

\DeclareMathOperator*{\argmin}{argmin}   	
   	
\DeclareMathOperator{\vol}{vol}  	
\DeclareMathOperator{\Int}{int}

\let\<=\langle
\let\>=\rangle

\def\...{...}
\newcommand{\shortStyle}{\textit}

\newcommand{\wrt}{\shortStyle{w.r.t.}}

\makeatletter
\renewcommand*{\eqref}[1]{%
  \hyperref[{#1}]{\textup{\tagform@{\ref*{#1}}}}%
}
\makeatother

\begin{document}


\expandafter\title[A relation between Mahler volume and flag number for polytopes]
{A relation between Mahler volume and flag number for convex polytopes}
		
\author[M.\ Winter]{Martin Winter$^\dagger$}
\address{$^\dagger$Max-Planck Institute for Mathematics in the Sciences, Inselstraße 22, 04103 Leipzig, Germany}
\email{martin.winter@mis.mpg.de}
\date{\today}
	
		
\date{\today}
% ...

\maketitle

\section{Introduction}

Let $P\subset\RR^d$ be a convex full-dimensional polytope.
A \emph{flag} $\bs F=\{F_0\subset F_1\subset\cdots$ $\subset F_d\}$ of $P$ is a maximal chain of faces of $P$.
Let $\mathcal F(P)$ denote the number of flags of $P$.
We prove the following:

\begin{theorem}
    \label{thm:main}
    For a polytope $P\subset\RR^d$ holds
    \begin{equation}
        \label{eq:main}
        \vol(P) \vol(P-P)^\circ \le \frac{\mathcal F(P)}{(d!)^2}
    \end{equation}
\end{theorem}

The expression $\vol(P)\vol(P-P)^\circ$ is the \emph{Makai volume}, which is conjectured to be bounded below by $(d+1)/d!$ \cite{aliev2025mahler,makai1978thinnest}.
\Cref{thm:main} in particular implies the following relation between Mahler volume and number of flag conjectured by Freij, Schmitt, Schymura and Ziegler \cite[Conjecture 10.1]{faifman2025volume}\cite{freij2012face}:

\begin{corollary}
    \label{res:cs}
    For a centrally symmetric polytope $P\subset\RR^d$ holds
    \begin{equation}
        \label{eq:cs}
        \vol(P)\vol(P^\circ) \le \frac{2^d}{(d!)^2} \mathcal F(P).
    \end{equation}
\end{corollary}
\begin{proof}
    We have $\vol (P-P)^\circ = \vol (2P)^\circ = 2^{-d}\vol(P^\circ)$.
    Hence
    \[\vol(P) \vol(P^\circ) = 2^d \vol(P) \vol(P-P)^\circ \overset{\text{\eqref{eq:main}}}\le \frac{2^d}{(d!)^2} \mathcal F(P).\qedhere\]
\end{proof}

In particular, assuming the Mahler bound $\vol(P)\vol(P^\circ)\ge 4^d/d!$, we get
$$\mathcal F(P) \overset{\text{\eqref{eq:cs}}}\ge \frac{(d!)^2}{2^d} \vol(P)\vol(P^\circ) 
\ge \frac{(d!)^2}{2^d}\cdot\frac{4^d}{d!} = 2^d d!.$$
which is precisely the bound predicted by \textit{Kalai's flag conjecture}.
In other words, \textit{Mahler's conjecture} for centrally symmetric polytopes implies the flag conjecture.
Such a direct relation between the conjectures has been suspected for a long time, but was not previously known.
More generally, any lower bound on the Mahler~volume now yields a lower bound on the flag number.
For example, Kuperberg's bound $\vol(P)\vol(P^\circ)\ge\pi^d/d!$ \cite{kuperberg2008mahler} yields
$$\mathcal F(P) \ge \Big(\frac\pi2\Big)^d d! \approx 1.5707^d d!.$$

The equality case remains unresolved in this note. 

\begin{conjecture}\,
    \label{conj:equality_case}
    \begin{myenumerate}
        \item Equality in \eqref{eq:cs} is attained exactly for the Hanner polytopes.
        \item Equality in \eqref{eq:main} is attained exactly for the Hanner polytopes and their faces.
    \end{myenumerate}
\end{conjecture}


\subsection{A comment on OpenAI's proof of Mahler's conjecture}

On October~6, 2026 OpenAI announced a solution of Mahler's conjecture obtained through use~of an internal frontier model \cite{OpenAI2026SymmetricMahler}.
They provide a writeup of the proof alongside a formalization in Lean.
Assuming this solution is correct, this article therefore provides a proof of Kalai's flag conjecture.

It is worth noting some parallels between the proof strategies applied to Mahler's conjecture and \cref{res:cs}.
Despite the apparent primal-dual symmetry in the~formulation of both conjectures, neither one treats primal and dual equally. Instead, in both proofs, one uses points $t\in P^\circ$ to parametrize certain subdivisions~of~$P$~(pulling triangulations in our case).
This then also induces a partition of $P^\circ$, where points $t_1,t_2\in P^\circ$ from the same region give the same subdivision of $P$. 
In our proof these regions are polytopal, while in the proposed proof of Mahler they are real analytic.

\section{Pulling triangulations}
\label{sec:pulling}

The proof of \cref{thm:main} uses so-called \emph{pulling triangulations}. 
These form a~type~of polytope triangulation that requires no new internal vertices (unlike, say, a barycentric triangulation).
We recall the relevant basics here (see, for example, \cite{deLoera2010triangulations,lee1997subdivisions}).

A pulling triangulation can be constructed from an order on the vertices.
Such an order induces an order on the vertices of each face, which we use to recursively define the pulling triangulation by induction over the faces.
Triangulating a vertex or edge is trivial and independent of the vertex order.
Suppose we already triangulated the $(k-1)$-faces. To triangulate a $k$-face $F$, let $v_i\in F$ be the vertex of minimal~index.
For each facet $G$ of $F$ and each already \mbox{constructed \emph{pulling simplex}~$\Delta\subseteq G$,\nls $\Delta\vee v_i$} (the cone over $\Delta$ with apex $v_i$) defines a pulling simplex of $F$.
Ignoring any~degenerate simplices, we obtain a triangulation of $F$.

There is a way to index the pulling simplices via the flags of $P$.
If $\bs F=\{F_0\subset\cdots$ $\subset F_d\}$ is a flag, let $v(F_k)$ be the vertex $v_i\in F_k$ with minimal index $i$.
One can then verify that
$$P=\dot{\bigcup_{\bs F}}\, \Delta_v(\bs F),
\qquad \text{where } \Delta_v(\bs F) := \conv\{v(F_0),...,v(F_d)\}$$
defines the same pulling triangulation of $P$ \wrt\ the given vertex order.
Note that some of the $\Delta_v(\bs F)$ might be degenerate.
Those will not be relevant later. 

\section{Proof of \cref{thm:main}}

Write $Q:=(P-P)^\circ$.
The proof idea is to subdivide $P\times Q$ into finitely many polytopal pieces. We then sort the pieces into $\mathcal F(P)$-many classes and show that~the total volume in each class is not more than $1/(d!)^2$. 
We first define the subdivision.

Let $\mathcal V$ be the vertex set of $P$.
For a (bijective) labeling $v\:\{1,...,n\}\to\mathcal V$, define
$$Q_v:= \{t\in Q\mid \<t,v_1\>\le \cdots \le \<t,v_n\>\}.$$
Each individual labeling $v$ also defines an order on the vertices, and hence a pulling triangulation of $P$ in the sense of \cref{sec:pulling}.
We obtain two subdivisions
$$Q=\dot{\bigcup_v}\, Q_v, \qquad P=\dot{\bigcup_{\bs F}}\, \Delta_v(\bs F),$$
with $\Delta_v(\bs F) := \conv\{v(F_0),...,v(F_d)\}$ defined as in \cref{sec:pulling}.
Note that for $t\in Q_v$ the order on the vertices is given by the linear form $\<t,\_\>$ and hence
\begin{align}
    \label{eq:order}
    \<t,v(F_0)\> \ge \cdots\ge \<t,v(F_d)\>.
\end{align}
%
%

Now we have
\begin{align*}
    \vol(P)\vol(Q) 
    &= \sum_v\sum_{\bs F} \vol(\Delta_v(\bs F))\vol(Q_v)
    \\&= \sum_{\bs F} \underbrace{\sum_v \vol(\Delta_v(\bs F))\vol(Q_v)}_{(*)\;\le 1/(d!)^2}
    \le \frac{\mathcal F(P)}{(d!)^2}
\end{align*}
and it remains to prove $(*)$.

\begin{lemma}
    For each flag $\bs F$ holds
    $$\sum_v \vol(\Delta_v(\bs F))\vol(Q_v)\le \frac1{(d!)^2}.$$
\end{lemma}
\begin{proof}
    We have
    $$
    \vol(\Delta_v(\bs F)) = \frac1{d!}
    \left|\smash{\,\det\!\underbrace{
    \begin{pmatrix}
        v(F_0) - v(F_1)
        \\
        \vdots
        \\
        v(F_{d-1}) - v(F_d)
    \end{pmatrix}}_{=:D_v}}
    \mathrlap{\phantom{
    \begin{matrix}
        x \\ x \\ x  
    \end{matrix}}}\right|
    .
    \mathrlap{\phantom{
    \underbrace{\begin{matrix}
        x \\ x \\ x  
    \end{matrix}}_D
    }}
    $$
    and therefore
    $$\sum_v\vol(\Delta_v(\bs F))\vol(Q_v)= \frac1{d!}\sum_v|\kern-1pt\det(D_v)|\vol(Q_v)=\frac1{d!} \sum_v\vol(D_v Q_v).$$
    Thus, it suffices to bound the right sum by $1/d!$.
    This follows with two ingredients:
    \begin{myenumerate}
        \item $D_vQ_v\subseteq \bs\Delta_d:=\{x\in\RR^d_{\ge0}\mid x_1+\cdots+x_d \le 1\}$, and
        \item if $v\ne w$ then $D_vQ_v$ and $D_w Q_w$ have disjoint interiors.
    \end{myenumerate}
    Then
    \begin{equation}
    \label{eq:le_simplex}
        \sum_v \vol(D_v Q_v) \overset{\text{\itm2}}= \vol\Big(\bigcup_v D_vQ_v\Big) \overset{\text{\itm1}}\le \vol(\bs\Delta_d) =
        \frac1{d!}.        
    \end{equation}

    Proof of \itm1:
    fix $t\in Q_v$. Then $(D_v t)_k = \<t,v(F_{k-1})\>-\<t,v(F_k)\> \ge0$ by \eqref{eq:order}.\nls
    Also, through telescoping we get
    \begin{equation}
    \label{eq:le_one}
        \sum_k (D_v t)_k = \<t,v(F_0)\>-\<t,v(F_d)\>\le 1,
    \end{equation}
    where the inequality follows from
    $$t\in Q=(P-P)^\circ = \big\{\,t\in\RR^d \;\big\vert\; \max_{x\in P}\<t,x\> - \min_{x\in P} \<t,x\> \le 1\,\big\}.$$

    Proof of \itm2:
    For sake of contradiction suppose that $D_vQ_v$ and $D_wQ_w$ intersect~in an interior point
    $D_vt_v=D_wt_w$. In particular, $D_v$ and $D_w$ are non-singular, and~$t_v\in \Int(Q_v)$ and $t_w\in \Int(Q_w)$.
    Unfolding the definition of the matrices yields
    \begin{equation}
        \label{eq:eq}
        \<t_v,v(F_{k-1})\>-\<t_v,v(F_k)\> = \<t_w,w(F_{k-1})\>-\<t_w,w(F_k)\>,\;\;\text{for $k\in\{1,...,d\}$}.
    \end{equation}
    
    Note that $v(F_0)=w(F_0)$, and we may translate $P$ so that $v(F_0)=w(F_0)=0$.
    Then the system \eqref{eq:eq} can be equivalently expressed as
    \begin{equation}
        \label{eq:eq2}
        \<t_v,v(F_k)\> = \<t_w,w(F_k)\>,\;\;\text{for $k\in\{0,...,d\}$}.
    \end{equation}

    We show that the linear forms $\<t_v,\_\>$ and $\<t_w,\_\>$ agree on $F_k$ for all $k$.
    This~is clear for $k=0$.
    Suppose then that they agree on $F_{k-1}$.
    Since on $F_k$ the vertex~$v(F_k)$ \mbox{minimizes~$\<t_v,\_\>$}, and the vertex $w(F_k)$ minimizes $\<t_w,\_\>$, we get
    %
    %
    $$
        \<t_v,w(F_k)\> \ge \<t_v,v(F_k)\> \overset{\eqref{eq:eq2}}= \<t_w,w(F_k)\>
    $$
    \vspace{-3ex}
    $$
        \<t_v,v(F_k)\> \overset{\eqref{eq:eq2}}= \<t_w,w(F_k)\> \le \<t_w,v(F_k)\>
    $$
    Since $v(F_k)$ and $w(F_k)$ lie on the same side of $F_{k-1}$ in $F_k$, we must have $\<t_v,\_\>=\<t_w,\_\>$ on both points, and actually, on all of $F_k$.
    %
    In particular, since $F_d$ is spanning, we get $t_v=t_w$.
    This means $t_v,t_w\in \Int(Q_v)\cap \Int(Q_w)$, which implies $v=w$. 
    \qedhere

\end{proof}

\small 

\vspace{1em}

\subsection*{Acknowledgments.}
I want to thank Raman Sanyal and Jan Stricker for the many discussions about 2-level polytopes.
This is where we rediscovered this upper bound, before we learned that it has already been conjectured by Freij, Schmitt, Schymura and Ziegler for general polytopes.
In the context of 2-level polytopes, pulling triangulations come up naturally, and they lead to the strategy that was eventually used here.

\subsection*{Funding}The author is funded by the SPP 2458 ``Combinatorial Synergies'' (project ID 539851419), funded by the Deutsche Forschungsgemeinschaft (DFG, German Research Foundation).

\subsection*{Statement on AI use}
Generative AI was used during the mathematical development of this article. Starting from a proof strategy supplied and guided by the author, it assisted in developing the proof. 
All ideas were substantially revised and simplified by the author, who determined its final structure and presentation.
The entire article is written by the author without AI assistance. The author takes full responsibility for the mathematical content.


\bibliographystyle{abbrv}
\bibliography{literature}

\end{document}